\documentclass[letterpaper, two column, 10 pt, conference]{ieeeconf}  

\IEEEoverridecommandlockouts                              
\usepackage[utf8]{inputenc}
\usepackage{amsmath,bm}
\usepackage{amsfonts}
\usepackage{amssymb}
\usepackage{graphicx}
\usepackage{amsfonts}
\usepackage[colorlinks,urlcolor=blue]{hyperref}
\usepackage{caption}
\usepackage{subcaption}
\usepackage{algorithm}
\usepackage{algpseudocode}
\usepackage{cases}
\usepackage{enumerate}
\usepackage{xcolor}
\usepackage{cite}
\usepackage{verbatim}
\usepackage[normalem]{ulem}

\newtheorem{definition}{Definition}
\newtheorem{Theorem}{Theorem}

\newtheorem{Lemma}{Lemma}
\newtheorem{Problem}{Problem}

\newtheorem{Remark}{Remark}
\newtheorem{Assumption}{Assumption}

\newcommand{\mc}[1]{\mathcal{#1}}
\newcommand{\mb}[1]{\mathbb{#1}}
\newcommand{\abs}[1]{\lVert{#1} \rVert}
\newcommand\numberthis{\addtocounter{equation}{1}\tag{\theequation}}

\usepackage{tikz}
	\usetikzlibrary{shapes,arrows}
	\tikzstyle{frame} = [draw, -latex]
	\tikzstyle{line} = [draw]
	\tikzstyle{line2} = [draw, dashdotted]
	\tikzstyle{line3} = [draw, dashed]
	\tikzstyle{line3UD} = [draw, dashed]
	\tikzstyle{place} = [circle, draw=black, fill=white, thick, inner sep=2pt, minimum size=1mm]
	\tikzstyle{place2} = [circle, draw=black, fill=black, thick, inner sep=2pt, minimum size=1mm]
	\tikzstyle{placeRed} = [circle, draw=red, fill=red, thick, inner sep=2pt, minimum size=1mm]
	\tikzstyle{vertex} = [circle, draw=black, fill=black, thick, inner sep=2pt, minimum size=1mm]
\usepackage{tkz-euclide}
\tikzset{
    right angle quadrant/.code={
        \pgfmathsetmacro\quadranta{{1,1,-1,-1}[#1-1]}     
        \pgfmathsetmacro\quadrantb{{1,-1,-1,1}[#1-1]}},
    right angle quadrant=1, 
    right angle length/.code={\def\rightanglelength{#1}},   
    right angle length=1.4ex, 
    right angle symbol/.style n args={3}{
        insert path={
            let \p0 = ($(#1)!(#3)!(#2)$) in     
                let \p1 = ($(\p0)!\quadranta*\rightanglelength!(#3)$), 
                \p2 = ($(\p0)!\quadrantb*\rightanglelength!(#2)$) in 
                let \p3 = ($(\p1)+(\p2)-(\p0)$) in  
            (\p1) -- (\p3) -- (\p2)
        }
    }
}

\title{\LARGE \bf
Direction-only Orientation Alignment of Leader-Follower Networks
}

\author{Quoc Van Tran, Hyo-Sung Ahn, and Jinwhan Kim
\thanks{Q. V. Tran and J.H. Kim are with the Department of Mechanical Engineering, Korea Advanced Institute of Science and Technology (KAIST), Daejeon, Republic of Korea. E-mails: {\tt\small $\{$quoctran;jinwhan$\}$@kaist.ac.kr}}
\thanks{H.-S. Ahn is with the School of Mechanical Engineering, Gwangju Institute of Science and Technology, Gwangju, Republic of Korea. E-mail: {\tt\small hyosung@gist.ac.kr}}
}

\begin{document}

\maketitle
\thispagestyle{empty}
\pagestyle{empty}

\begin{abstract}
When a team of agents, such as unmanned aerial/underwater vehicles, are operating in $3$-dimensional space, their coordinated action in pursuit of a cooperative task generally requires all agents to either share a common coordinate system or know the orientations of their coordinate axes with regard to the global coordinate frame. Given the coordinate axes that are initially unaligned, this work proposes an orientation alignment scheme for multiple agents with a type of leader-following graph typologies using only inter-agent directional vectors, and the direction measurements to one or more landmarks of the first two agents. The directional vectors are expressed in the agents' body-fixed coordinate frames and the proposed alignment protocol works exclusively with the directional vectors without the need of a global coordinate frame common to all agents or the construction of the agents' orientation matrices. Under the proposed alignment scheme, the orientations of the agents converge almost globally and asymptotically to the orientation of the leader agent. Finally, numerical simulations are also given to illustrate the effectiveness of the proposed method.
\end{abstract}

\section{Introduction}
Distributed coordination control of a group of multiple agents has attracted much research interest over the past few decades \cite{Oh2015survey, HSAhn2019,ZFeng2020Review, ZPeng2021TI}. This is due to the enhanced efficacy, scalability, robustness, and ability to perform complex group-level missions by the deployment of a team of agents rather than a single agent. Therefore, distributed coordination control over multi-agent systems finds applications in various engineering disciplines including formation control \cite{Oh2015survey}, sensor network localization/synchronization \cite{HSAhn2019,Quoc2020auto}, and cooperative exploration and manipulation \cite{ZFeng2020Review, ZPeng2021TI}. When multiple mobile agents are operating in $3$-dimensional space, their coordinated action in pursuit of a cooperative task generally requires all agents to either share a common coordinate frame or know the orientations of their body-fixed coordinate axes with regard to the global (north-east-down) reference frame. If the agents' coordinate axes are initially unaligned, either
orientation estimation \cite{Tron2014tac, Lee2017, Quoc2019necsys, Quoc2018cdc, Quoc2018tcns,Leonardos2019cdc} or orientation alignment (or synchronization) schemes \cite{Igarashi2009, Sarltte2009, Thunberg2017,Markdahl2018tac, Gui2018auto, Zong2016, JWei2018tac, Thunberg2016} may be employed so that the agents can compensate for the misalignment of their local coordinate systems.

In the $2$-dimensional plane, if each agent measures the directions, i.e., the bearing angles, to its neighboring agents, then two neighboring agents can easily compute the relative orientation angle between their coordinate systems \cite{Oh2014tac}. 
For a system of multiple agents, using the relative angles between the coordinate axes of the agents, distributed orientation alignment \cite{Oh2014tac} and orientation estimation \cite{Lee2016auto} on the circle are proposed, respectively. However, in $3$-dimensional space, a pair of direction measurements between two neighboring agents are insufficient for the two agents to determine their \textit{relative orientation}, i.e., a rotation matrix in the Special Orthogonal group $SO(3)$, between their local coordinate systems due to the flexibility of the rotation along the common direction from one to the other agent \cite{Quoc2018cdc}. This can be overcome by examining additional direction constraints of each of the two agents to a third agent or a landmark that they both observe. Indeed, as shown in \cite{Quoc2018cdc,Leonardos2019cdc}, by exploiting the triangulation sensing network, two neighboring agents can compute their relative orientation matrix. The agents' orientations then can be computed by using a consensus-based protocol up to a common rotation \cite{Quoc2018cdc} or a Riemannian gradient descent algorithm \cite{Leonardos2019cdc}. Distributed orientation estimation based on inter-agent direction measurements in leader-follower networks with two leaders aware of their actual orientations was proposed in \cite{Quoc2020auto}.

The distributed alignment control of the agents' local coordinate frames (also known as synchronization on $SO(3)$) has been investigated in \cite{Igarashi2009, Sarltte2009, Thunberg2017,Markdahl2018tac, Gui2018auto, Zong2016, JWei2018tac, Thunberg2016} based on relative orientations between the agents. 
Due to the topological obstruction associated with the rotation matrix set $SO(3)$ \cite{Bhat2000SCL}, (distributed) continuous orientation consensus laws on $SO(3)$ can ensure only local \cite{Igarashi2009,Moshtagh2009tro,Oh2014tac, Thunberg2014} or at most almost global stability \cite{Markdahl2018tac}. 
To achieve almost global synchronization on $SO(3)$, \cite{Markdahl2018tac} proposed a consensus protocol by combining two (almost) global consensus schemes that are designed for the circle and $2$-sphere, respectively. Global orientation consensus on $SO(3)$ can be achieved using quaternion-based hybrid feedback laws \cite{Mayhew2012tac, Gui2018auto}.  A finite-time leader-following orientation consensus scheme was presented in \cite{Gui2018auto} based on the distributed observation of the leader's orientation and the quaternion representation of rotations. However, since the local representations of rotations, e.g., the unit-quaternion \cite{Gui2018auto} or the angle-angle representation \cite{Thunberg2014}, double cover the $SO(3)$ space, control protocols using these local representations may experience the undesirable unwinding phenomenon\cite{Bhat2000SCL}.

In this work, we aim to investigate the orientation alignment for multi-agent systems with directed graph typologies based only on the measurements of inter-agent directional vectors and the direction measurements to one or more landmarks (or the third agent if no landmarks are available) of the first two agents. 
The orientation alignment problem under study is motivated by the collective behaviors in nature, such as the flocking of birds and schooling of fishes. In a navigation task, there are several leaders which can sense directions to objects in the outside environments, and a number of other follower agents which track the leaders by only sensing directions to their neighboring agents. 
We note that the leader-follower types of structures have been extensively studied in the field of (distributed) networked coordination control \cite{Oh2015survey, ZPeng2021TI, HSAhn2019}. For example, in coordination control of autonomous surface vessels (ASVs), it is often desired to have a leader which guilds the motion of the system and follower agents which follow the leader, while the whole system maintains a certain formation pattern \cite{ZPeng2021TI}. 

The specific contributions of this work are as follows. First, we propose an orientation alignment scheme for multiple agents with a leader-following graph typology based on only inter-agent directional vectors and the direction measurements to one or more landmarks (or the third agent) of the first two agents. The proposed orientation control protocol for each agent is in the form of a gradient-based control law associated with an error function which is a weighted sum of the misalignment of the directional vectors measured by the agent and its neighbors. Thus, the proposed alignment protocol works exclusively with the directional vectors with no need for a global coordinate frame common to all agents or the computation of (relative) orientation matrices, as opposed to \cite{Quoc2018cdc,Leonardos2019cdc,Igarashi2009, Sarltte2009, Thunberg2017,Markdahl2018tac, Gui2018auto, Zong2016, JWei2018tac, Thunberg2016}. The proposed orientation alignment scheme is an extension of the orientation localization law in \cite{Quoc2020auto} to the orientation control problem. In addition, in contrast to \cite{Quoc2020auto}, the requirement of two leaders that are aware of their true orientations is relaxed by utilizing their direction measurements to nearby landmarks (or a common neighbor). Second, the equilibrium points of the orientation control system are characterized, in which the steady-state orientation matrices of the agents constitute the critical points of the associated error functions. Further, we show that the orientations of all follower agents converge almost globally and asymptotically to the orientation of the leader agent. Finally, numerical simulations are given to support and illustrate the theoretical development.

The remainder of this paper is outlined as follows. Section \ref{sec:preliminary} presents preliminaries and formulates the orientation alignment problem. The orientation alignment laws are proposed  and an almost global stability analysis is established in Section \ref{sec:orientation_align}. Section \ref{sec:simulation} provides simulation results. Finally, Section \ref{sec:Conclusion} concludes this paper.

\section{Preliminaries and Problem Formulation}\label{sec:preliminary} 
\subsubsection*{Notation} 
The dot and cross products are denoted by $\cdot$ and $\times$, respectively. The symbol $\Sigma$ represents a global coordinate frame and the symbol $^k\Sigma$ with superscript $k$ denotes the $k$-th local coordinate frame. A vector $\bm{x}\in \mb{R}^3$ expressed in $^k\Sigma$ and $\Sigma$ are denoted as $\bm{x}^k$ and $\bm{x}$, respectively. Let $\bm{1}_n=[1,\ldots,1]^\top\in \mb{R}^n$ be the vector of all ones, and $\bm{I}_3$ denotes the $3\times 3$ identity matrix. 
The trace of a matrix is denoted by $\text{tr}(\cdot)$. 
The set of rotation matrices and orthogonal matrices in $\mb{R}^3$ are denoted by $SO(3)\}$ and $O(3)$, respectively. For a symmetric matrix $\bm{X}$, $\bm{X}\succ 0$ implies that $\bm{X}$ is positive semidefinite and $\lambda(\bm{X})$ denotes the set of its eigenvalues.

We denote the set of $3\times 3$ skew-symmetric as $\mathfrak{so}(3):=\{\bm{A}\in \mb{R}^{3\times 3}|\bm{A}^\top=-\bm{A}\}$. For any $\omega\in \mb{R}^3$, the \textit{hat} map $(\cdot)^{\wedge}:~\mb{R}^3\rightarrow \mathfrak{so}(3)$ is defined such that $\omega\times \bm{v}=\omega^\wedge\bm{v},\forall \bm{v}\in \mb{R}^3$.
The \textit{vee} map is the inverse of the \textit{hat} map and defined as $(\cdot)^\vee:~\mathfrak{so}(3)\rightarrow \mb{R}^3$. The \textit{exponential map} $exp:\mathfrak{so}(3)\rightarrow SO(3)$ is \textit{surjective} and $T_{\bm{R}}SO(3)=\{\bm{R}\eta^\wedge:\eta^\wedge\in \mathfrak{so}(3)\}$ denotes the tangent space at a point $\bm{R}\in SO(3)$.

For any $\bm{x},\bm{y},\bm{z}\in \mb{R}^3$, $\bm{A},\bm{B}\in \mb{R}^{3\times 3}$, and $\bm{R}\in SO(3)$ we have the following relations \cite{Bullo2005spr, Mahony2008tac}.
\begin{align}
&\qquad\bm{x}\times \bm{y}=-\bm{y}\times \bm{x}  \label{eq:cross_prod_1} \\
&(\bm{R}\bm{x})\times(\bm{R}\bm{y})=\bm{R}(\bm{x}\times\bm{y}),~\bm{R}\bm{x}^\wedge\bm{R}^\top=[\bm{Rx}]^\wedge \label{eq:cross_prod_2}\\
&\qquad (\bm{x}\times\bm{y})^\wedge=\bm{x}^\wedge\bm{y}^\wedge-\bm{y}^\wedge\bm{x}^\wedge=\bm{y}\bm{x}^\top-\bm{x}\bm{y}^\top\label{eq:cross_prod_3}\\
&\qquad\bm{x}\cdot\bm{y}^\wedge\bm{z}=\bm{z}\cdot\bm{x}^\wedge\bm{y} =\bm{y}\cdot\bm{z}^\wedge\bm{x}\label{eq:cross_prod_4}\\
&\qquad \bm{x}\times(\bm{y}\times \bm{z})+\bm{y}\times(\bm{z}\times \bm{x})+\bm{z}\times(\bm{x}\times \bm{y})=0 \label{eq:Jacobi_indentity}\\
&\qquad\bm{x}\cdot\bm{y}=\bm{x}^\top\bm{y} = \text{tr}(\bm{x}\bm{y}^\top) \label{eq:trace_property_1}\\
&\qquad \text{tr}(\bm{A}\bm{B})=\text{tr}(\bm{B}\bm{A})=\text{tr}(\bm{A}^\top\bm{B}^\top)\label{eq:trace_property_3}
\end{align}

\subsection{Graph Theory}
An interaction graph of a multi-agent network is denoted by $\mc{G}=(\mc{V},\mc{E})$, where, $\mc{V}=\{1,\ldots,n\}$ denotes the vertex set and $\mc{E}\subseteq\mc{V}\times \mc{V}$ denotes the set of edges of $\mc{G}$. An edge is defined by the ordered pair $e_k=(i,j), k=1,\ldots,m, ~i,j\in \mc{V},~i\neq j,$ with $m=\vert \mathcal{E} \vert$ being the number of edges. The graph $\mc{G}$ is said to be \textit{undirected} if $(i,j)\in \mc{E}$ implies $(j,i)\in \mc{E}$, or equivalently, $j$ and $i$ are neighbors of each other. If the graph $\mc{G}$ is directed, $(i,j)\in \mc{E}$ does not necessarily imply $(j,i)\in \mc{E}$. The set of neighboring agents of $i$ is given by $\mc{N}_i=\{j\in\mc{V}:(i,j)\in \mc{E}\}$. 

\subsection{Problem formulation}
We consider the system of $n$ stationary agents and some non-colocated landmarks, e.g., features in the environment, in the three-dimensional space. Associated with each agent $i$, there are a position vector, $\bm{p}_i\in \mb{R}^3$, taken at its centroid, and a body-fixed coordinate frame, $^i\Sigma$. The orientation of the local coordinate frame $^i\Sigma$ relative to the global coordinate frame is denoted as $\bm{R}_i\in SO(3)$. 
We define the directional vector from an agent $i$ to a neighbor $j$ as
\begin{equation}\label{eq:direction_vector}
\bm{b}_{ij}^i = \bm{R}_i^\top\frac{\bm{p}_j-\bm{p}_i}{||\bm{p}_j-\bm{p}_i||}=\bm{R}_i^\top\bm{b}_{ij},
\end{equation}
and similarly, the directional vector pointing from agent $j$ to agent $i$ is $\bm{b}_{ji}^j=\bm{R}_j^\top\bm{b}_{ji}$, where in the global coordinate frame $\bm{b}_{ji}=\frac{\bm{p}_i-\bm{p}_j}{||\bm{p}_i-\bm{p}_j||}=-\bm{b}_{ij}$.
\begin{figure}[t]
\centering
\begin{tikzpicture}[scale=1.4]
\node (pz) at (0,1,0) []{};
\node (bij) at (1.15,0.5,0) {};
\node (bji) at (1.2,.5,0) {};
\node (p_j) at (2.5,1,0) [label=below:$2$]{};
\node[place] (p_i) at (0,0,0) [label=left:$1$]{};
\node (w_i) at (-0.4,.6,0) [label=left:${^1\Sigma}$]{};
\node (Sigj) at (2.4,1.7,0) [label=right:${^2\Sigma}$]{};
\node[place, black] (x) at (.3,1.6) [label=left:$x$]{};
\draw[{line width=0.7pt},blue] (0,0,0) [frame] -- (0.7,0,0);
\draw[{line width=0.7pt},blue,->] (0,0,0)[frame]   -- (w_i);
\draw[{line width=0.7pt},blue] (0,0,0) [frame] -- (pz);
\draw[{line width=0.7pt},blue] (p_j) [frame] -- (2.1,1.4,0);
\draw[{line width=0.7pt},blue,->] (p_j)[frame]  -- (3,1.4,0);
\draw[{line width=0.7pt},blue] (p_j) [frame] -- (2.5,1.8,0);
\draw[shorten >=1.9cm, ->,line width=1.pt] (p_i) [frame] -- (p_j) node [pos =0.4, yshift=-0.3ex, below] {$\bm{b}_{12}^1$};
\draw[shorten >=1.9cm, ->,line width=1.pt] (p_j) [frame] -- (p_i) node [pos =0.2, yshift=-0.3ex, below] {$\bm{b}_{21}^2$};
\draw[shorten >=0.75cm, ->,line width=1.pt] (p_i) [frame] -- (x) node [pos =0.5, yshift=-0.ex, right] {$\bm{b}_{1x}^1$};
\node (b2x) at (1.3,1.3,0) {};
\draw[shorten >=1.1cm, ->,line width=1.pt] (p_j) [frame] -- (x) node [pos =.5, yshift=-0.2ex, above right] {$\bm{b}_{2x}^2$};
\draw[{line width=.7pt},black, dashed] (0,0,0)  -- (p_j);
\draw[{line width=.7pt},black, dashed] (0,0,0)  -- (x);
\draw[{line width=.7pt},black, dashed] (p_j)  -- (x);
\node[place] () at (p_j) []{};
\end{tikzpicture}
\caption{Agents $1$ and $2$ measure the directional vector $\bm{b}_{ix}^i,i=1,2$ to a common landmark $x\in \mc{V}_a$, and the directions to each other $\bm{b}_{12}^1$  and $\bm{b}_{21}^2$, respectively.} 
\label{fig:agents_1_2}
\end{figure}
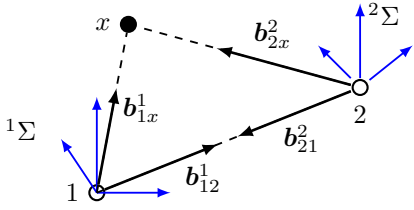
\begin{figure}[t]
\centering
\begin{subfigure}[t]{0.23\textwidth}
\centering
\begin{tikzpicture}[scale=1.4]
\node[place] (p1) at (0,0) [label=left:$1$]{};
\node[place] (p2) at (1,0) [label=right:$2$]{};
\node[place] (p3) at (0.5,-0.75) [label=below right:$3$]{};
\node[place] (p4) at (-0.5,-0.75) [label=left:$4$]{};
\node[place] (p5) at (1.5,-0.75) [label=right:$5$]{};
\node[place, black] (x1) at (0.25,0.5) [label=left:$x_1$]{};
\node[place, black] (x2) at (.9,0.5) [label=right:$x_2$]{};

\draw[line width=1pt,->] (p3)[frame]  -- (p1);
\draw[line width=1pt,->] (p3)[frame]  -- (p2);
\draw[line width=1pt,->] (p4)[frame]  -- (p3);
\draw[line width=1pt,->] (p4)[frame]  -- (p1);
\draw[line width=1pt] (p2)  -- (p1);
\draw[line width=1pt,->] (p5)[frame]  -- (p3);
\draw[line width=1pt,->] (p5)[frame]  -- (p2);
\draw[line width=1pt,->, dashed] (p1)[frame]  -- (x1);
\draw[line width=1pt,->, dashed] (p2)[frame]  -- (x1);
\draw[line width=1pt,->, dashed] (p1)[frame]  -- (x2);
\draw[line width=1pt,->, dashed] (p2)[frame]  -- (x2);
\draw[line width=1pt,->] (p2)[frame]  -- (p1);
\end{tikzpicture}
\caption{$\mc{G}=(\mc{V},\mc{E})$}
\label{subfig:info_flow}
\end{subfigure}
\begin{subfigure}[t]{0.23\textwidth}
\centering
\begin{tikzpicture}[scale=1.4]
\node[place] (p1) at (0,0) [label=left:$1$]{};
\node[place] (p2) at (1,0) [label=right:$2$]{};
\node[place] (p3) at (0.5,-0.75) [label=below right:$3$]{};
\node[place] (p4) at (-0.5,-0.75) [label=left:$4$]{};
\node[place] (p5) at (1.5,-0.75) [label=right:$5$]{};

\draw[line width=1pt] (p3)  -- (p1);
\draw[line width=1pt] (p3)  -- (p2);
\draw[line width=1pt] (p4)  -- (p3);
\draw[line width=1pt] (p4)  -- (p1);
\draw[line width=1pt] (p2)  -- (p1);
\draw[line width=1pt] (p5) -- (p3);
\draw[line width=1pt] (p5) -- (p2);
\draw[line width=1pt] (p1) -- (p2);
\end{tikzpicture}
\caption{Sensing graph}
\label{subfig:direction_graph}
\end{subfigure}
\caption{A leader-follower network and two landmarks in $\mb{R}^3$ (a). The underlying undirected graph of the network characterizes the direction sensing between the agents (b).}
\label{fig:leader_follower_network}
\end{figure}
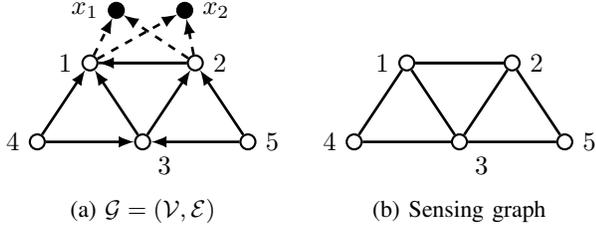
The \textit{leader-follower} system studied in this work is constructed as follows.
\begin{definition}\label{def:leader_follower}
A \textit{leader-follower} network is a directed network whose agents are ordered such that: (a) the first agent $1$ is the only neighbor of agent $2$; they locally measure the directions between them, i.e., $\{\bm{b}_{12}^1,\bm{b}_{21}^2\}$, and directional vectors to nearby landmarks (see e.g., Fig. \ref{fig:agents_1_2}); (b) an agent $i,~ 3\leq i \leq n,$ has two (or more) neighboring agents $j\in\{1,\ldots,i-1\}$. Agent $i$ knows the direction $\bm{b}^i_{ij}$ to the neighbor $j$, while its neighbor knows the direction $\bm{b}^j_{ji}$ to agent $i$.
\end{definition}

We denote by $\mc{V}_a:=\{x_1,x_2,\ldots\}$ the set of landmarks. The interactions of the leader-follower network is characterized by an (\textit{acyclic}) directed graph $\mc{G}=(\mc{V},\mc{E})$.  
Agent $1$ need not know its actual orientation matrix nor the global coordinate frame. In addition, we aim to align the other agents' orientations to the orientation of agent $1$, i.e., $\bm{R}_1$, as will be studied in Section \ref{sec:orientation_align}. Therefore, agent $1$, whose orientation is fixed, will be referred to as the \textit{leader}; the other agents in $\{2,\ldots,n\}$ will be called \textit{follower agents}. We will also assume that each follower has precisely two neighbors, for simplicity and the minimal connectivity of the sensing graph.

The rotational dynamics of agent $i$ is given by following kinematic equation:
\begin{equation}\label{eq:rotation_dynamics}
\dot{\bm{R}}_i=\bm{R}_i(\bm\omega_i^i)^\wedge,
\end{equation}
where $\omega_i^i$ is the angular velocity of agent $i$, expressed in $^i\Sigma$. We adopt the following assumptions.
\begin{Assumption}\label{ass:direction_measurements}
If there is a directed edge $(i,j)\in \mc{E}$, agents $i$ and $j$ measure the directional vectors $\bm{b}_{ij}^i$ and $\bm{b}_{ji}^j\in \mb{R}^3$ with respect to their local coordinate frames, respectively. In addition, agent $i$ can also received information communicated from neighboring agent $j\in \mc{N}_i$.
\end{Assumption}

Thus, the directions of the edges in the graph $\mc{G}$ specify the information flow in the network (see Fig. \ref{subfig:info_flow}). In addition, since the direction measurements between any two neighboring agents is bidirectional, the underlying (undirected) graph of the network, which can be verified to be bearing-rigid \cite{Minh2020tac},
characterizes the direction sensing over the network (see Fig. \ref{subfig:direction_graph}).
\begin{Assumption}\label{ass:non_colocated}
Agents $1$ and $2$ are not collinear with each $x_i\in\mc{V}_a$, and all landmarks and agents $1$ and $2$ are noncoplanar. No two agents are collocated, and each follower agent $i\in \mc{V}\setminus\{1,2\}$ and its two neighbors are not collinear.
\end{Assumption}

Assumption \ref{ass:non_colocated} requires that the positions of the first two agents and the landmarks, and the positions of each agent and its two or more neighbors are \textit{generic}. This assumption together with the rigidity of the direction sensing graph are used to guarantee the solvability of the orientation alignment problem. Further, it can be verified that to ensure the non-coplanarity of the landmarks and agents $1$ and $2$, there should be two or more landmarks, to which agents $1$ and $2$ sense the directional vectors.

The orientation alignment problem setup is motivated by the collective behaviors in nature, such as the flocking of birds and schooling of fish. When they perform a navigation task, there are leaders that can sense directions to objects or landmarks in the outside environment, and there are followers that track the leaders by sensing only directions to their neighboring agents. However, note importantly that our proposed method still works for the cases with only one or even no available landmark using the triangulation network of the first three agents (see Remark \ref{rmk:no_landmark_avail}). 
We can now state the main problem studied in this paper.
\begin{Problem}\label{prob:orient_syn}
Consider a leader-follower network defined in Definition \ref{def:leader_follower} of $n$ stationary agents using only inter-agent directional vectors. Under Assumptions \ref{ass:direction_measurements} and \ref{ass:non_colocated}, design a control law for each agent such that the orientations of all agents reach a consensus asymptotically.
\end{Problem}
\section{Orientation Alignment}\label{sec:orientation_align}
In this section, we first present orientation alignment scheme for the first follower agent $2$ based on the measured directions between agents $1$ and $2$ and their directions to the common landmarks. The orientation alignment laws for the other agents are then investigated.
\subsection{The first follower agent}
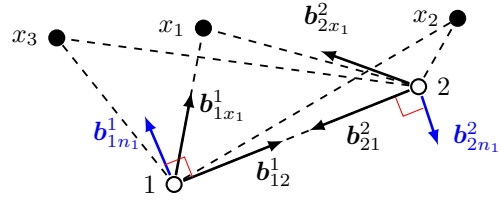
\begin{figure}[t]
\centering
\begin{tikzpicture}[scale =1.3]
\node (1) at (0,0,0) []{};
\node (bij) at (1.15,0.5,0) {};
\node (bji) at (1.2,.5,0) {};
\node (2) at (2.5,1,0) [label=right:$2$]{};
\node[place, black] (x) at (.3,1.6) [label=left:$x_1$]{};
\node[place, black] (y) at (2.9,1.7) [label=left:$x_2$]{};
\node[place, black] (x3) at (-1.2,1.5) [label=left:$x_3$]{};
\draw[shorten >=1.8cm, ->,line width=1.pt] (1) [frame] -- (2) node [pos =0.4, yshift=-0.3ex, below] {$\bm{b}_{12}^1$};
\draw[shorten >=1.8cm, ->,line width=1.pt] (2) [frame] -- (1) node [pos =0.2, yshift=-0.3ex, below] {$\bm{b}_{21}^2$};
\draw[shorten >=0.75cm, ->,line width=1.pt] (1) [frame] -- (x) node [pos =0.5, yshift=-0.ex, right]  {$\bm{b}_{1x_1}^1$};
\draw[{line width=1.pt}, blue] (1) [frame] -- (-0.3, 0.7) node [pos =0.7, yshift=-0.ex, left] {$\bm{b}_{1n_1}^1$};
\draw [right angle symbol={1}{2}{-0.3, 0.7}, red];
\draw[{line width=1.pt}] (2) [frame] -- (b2x) node [pos =1.3, yshift=-0.2ex, above right] {$\bm{b}_{2x_1}^2$};
\draw[{line width=1.pt}, blue] (2) [frame] -- (2.7,0.4) node [pos =1.3, yshift=-0.2ex, above right] {$\bm{b}_{2n_1}^2$};
\draw [right angle symbol={2}{1}{2.7,0.4}, red];
\draw[{line width=.7pt},black, dashed] (1)  -- (2);
\draw[{line width=.7pt},black, dashed] (1)  -- (x);
\draw[{line width=.7pt},black, dashed] (2)  -- (x);
\draw[{line width=.7pt},black, dashed] (1)  -- (y);
\draw[{line width=.7pt},black, dashed] (2)  -- (y);
\draw[{line width=.7pt},black, dashed] (1)  -- (x3);
\draw[{line width=.7pt},black, dashed] (2)  -- (x3);
\node[place] (p_i) at (0,0,0) [label=left:$1$]{};
\node[place] () at (2) []{};
\end{tikzpicture}
\caption{Illustration of the computation of directions of agents $1$ and $2$ to some landmarks $x_j$ (black nodes).} 
\label{fig:control_law_agents12}
\end{figure}
It is first noted that agents $1$, $2$ and any landmark $x_1\in \mc{V}_a$ form a plane in $\mb{R}^3$. Using the directional measurements $\{\bm{b}_{12}^1,\bm{b}_{1x_1}^1,\bm{b}_{21}^2,\bm{b}_{2x_1}^2\}$ (see also Fig. \ref{fig:agents_1_2}), each agent $i\in \{1,2\}$ additionally computes the following unit vector $\bm{b}_{in_1}^i$ corresponding to the landmark $x_1$:
\begin{align}
\bm{b}_{in_1}^i &=\frac{\bm{b}_{ij}^i\times \bm{b}_{ix_1}^i}{||\bm{b}_{ij}^i\times \bm{b}_{ix_1}^i||},\label{eq:agent12_normal_direct}
\end{align}
which is perpendicular to the plane ($1,2,x_1$) (see Fig. \ref{fig:control_law_agents12}). Thus, $\bm{b}_{1n_k}+\bm{b}_{2n_k}=0$ for all $x_k\in \mc{V}_a$. In addition, $\bm{b}_{12}=-\bm{b}_{21}\Longleftrightarrow\bm{R}_1\bm{b}_{12}^1=-\bm{R}_2\bm{b}_{21}^2\Longleftrightarrow\bm{b}_{12}^1=-\bm{R}_{12}\bm{b}_{21}^2$, where $\bm{R}_{12}:=\bm{R}_1^\top\bm{R}_2$ denotes the relative orientation between the local coordinate frames $^1\Sigma$ and $^2\Sigma$. Obviously, when the two local coordinate frames are aligned we have that $\bm{b}_{12}^1+\bm{b}_{21}^2=\bm{0}$ and $\bm{b}_{1n_k}^1+\bm{b}_{2n_k}^2=0$ for all $x_k\in \mc{V}_a$.

The time derivative of a directional vector $\bm{b}_{21}^2$ is given as
\begin{equation}\label{eq:derivative_of_direction}
\dot{\bm{b}}_{21}^2=\frac{d}{dt}(\bm{R}_2^\top\bm{b}_{21})=-({\bm\omega_2^2})^\wedge\bm{R}_2^\top\bm{b}_{21}=\bm{b}_{21}^2\times {\bm\omega_2^2}.
\end{equation}
Similarly, for each $x_j\in\mc{V}_a$, one also has $\dot{\bm{b}}_{2n_j}^2={\bm{b}}_{2n_j}^2\times \bm\omega_2^2$. Since the first agent does not rotate (i.e., $\bm\omega_1^1=\bm{0}$), $\dot{\bm{b}}^1_{1k}=\bm{0}$ for any $k$.

In the sequel, we define an error function as a weighted sum of squares of the misalignment of the measured directions by agents $1$ and $2$. We then characterize its critical points, at which the time derivative of the error function equals zero. It will be shown that the error function achieves the minimum at the desired aligned orientation (i.e., $\bm{R}_{1}^\top\bm{R}_2=\bm{I}_3$) and can be minimized using our proposed orientation alignment law.

\subsubsection{Error function}
Consider the error function associated with agent $2$ as
\begin{align*}\label{eq:error_function12}
\Phi_{2}&=\frac{1}{2}k_{21}||\bm{b}_{12}^1+\bm{b}_{21}^2||^2+\frac{1}{2}\sum_{x_l\in\mc{V}_a}k_{2x_l}||\bm{b}_{1n_l}^1+\bm{b}_{2n_l}^2||^2\\
&=k_{21}(1+\bm{b}_{12}^1\cdot\bm{b}_{21}^2)+\sum_{x_l\in\mc{V}_a}k_{2x_l}(1+\bm{b}_{1n_l}^1\cdot\bm{b}_{2n_l}^2),\numberthis
\end{align*}
where $k_{2j}>0,\forall j\in \mc{N}_1\cup\mc{V}_a$, are positive scalars. 
\subsubsection{Critical points of the error function}
To study the critical points of the error function $\Phi_{2}$, we rewrite it as follows:
\begin{align*}
\Phi_{2}&=k_{21}+\sum_{x_l\in\mc{V}_a}k_{2x_l}+k_{21}(\bm{R}_1^\top\bm{b}_{12})^\top(\bm{R}_2^\top\bm{b}_{21})\\
&+\sum_{x_l\in\mc{V}_a}k_{2x_l}(\bm{R}_1^\top\bm{b}_{1n_l})^\top(\bm{R}_2^\top\bm{b}_{2n_l})
\\
&=k_{21}+\sum_{x_l\in\mc{V}_a}k_{2x_l}-\mathrm{tr}(k_{21}\bm{R}_1^\top\bm{b}_{21}\bm{b}_{21}^\top\bm{R}_2)\\
&-\sum_{x_l\in\mc{V}_a}\mathrm{tr}(k_{2x_l}\bm{R}_i^\top\bm{b}_{2n_l}\bm{b}_{2n_l}^\top\bm{R}_2)\\
&=k_{21}+\sum_{x_l\in\mc{V}_a}k_{2x_l}-\mathrm{tr}(\tilde{\bm{Q}}_{2}\bm{K}_{2}), \numberthis \label{eq:error_function_rewrite1}
\end{align*}
where the second equality is derived using \eqref{eq:trace_property_1} and $\bm{b}_{1l}=-\bm{b}_{2l}$, and the last equality follows from \eqref{eq:trace_property_3}. Here, the matrices $\tilde{\bm{Q}}_{2}$ and $\bm{K}_{2}$ are respectively defined as
$$\tilde{\bm{Q}}_{2}:=\bm{R}_2\bm{R}_1^\top,~\bm{K}_{2}:=k_{21}\bm{b}_{21}\bm{b}_{21}^\top+\sum_{x_l\in\mc{V}_a}k_{2x_l}\bm{b}_{2n_l}\bm{b}_{2n_l}^\top.$$ 
\begin{Lemma}\label{lm:almost_surely_distinct}
Assume that Assumptions \ref{ass:direction_measurements} and \ref{ass:non_colocated} hold. The matrix $\bm{K}_{2}\in \mb{R}^{3\times 3}$ in \eqref{eq:error_function_rewrite1} is positive definite and has distinct positive eigenvalues for almost all positive scalars $k_{21}$ and $\{k_{2x_l}\},x_l\in\mc{V}_a$.
\end{Lemma}
\begin{proof} See Appendix \ref{app:almost_surely_distinct}.
\end{proof}
\begin{Remark}\label{rmk:no_landmark_avail}
The proof of the preceding lemma is based on the non-coplanarity assumption (Assumption \ref{ass:non_colocated}) for which the direction set $\{\bm{b}_{21},\bm{b}_{2n_l}\}_{x_l\in\mc{V}_a}$ spans the $\mb{R}^3$. If only one landmark is available, say $x_1$, the (normalized) cross product of the two directions $\bm{b}_{21}$ and $\bm{b}_{2n_1}$ (correspondingly, the cross product of $\bm{b}_{12}$ and $\bm{b}_{1n_1}$), which is perpendicular to both vectors, can be considered as the third direction constraint. Importantly, this computed direction together with $\bm{b}_{21}$ and $\bm{b}_{2n_1}$ span the $\mb{R}^3$. If no landmark is available, the direction measurements in the triangulation network $\{1,2,3\}$ can be used. Therefore, our proposed orientation alignment protocol for agent $2$ in the below would be applicable for the two cases.
\end{Remark} 

In the light of Lemma \ref{lm:almost_surely_distinct}, we can write $\bm{K}_{2}=\bm{U}\bm{G}_{2}\bm{U}^\top$ where $\bm{G}_{2}=\text{diag}\{\lambda(\bm{K}_{2})\}$ and $\bm{U}\in O(3)$. It can be shown that $\text{tr}(\bm{G}_{2})=\text{tr}(\bm{K}_{2})=k_{21}+\sum_{l\in\mc{V}_a}k_{2l}$. Thus, we can further rewrite $\Phi_{2}$ as:
\begin{align*}
\Phi_{2}&=\text{tr}(\bm{G}_{2}-\tilde{\bm{Q}}_{2}\bm{U}\bm{G}_{2}\bm{U}^\top)\\
&=\text{tr}(\bm{G}_{2}-\bm{U}^\top\tilde{\bm{Q}}_{2}\bm{U}\bm{G}_{2})\\
&=\text{tr}\big(\bm{G}_{2}(\bm{I}_3-\bm{U}^\top\tilde{\bm{Q}}_{2}\bm{U})\big).
\end{align*}
Then, we have the following lemmas:
\begin{Lemma}\cite[Prop. 11.31]{Bullo2005spr}\label{lm:stationary_points}
Let $\bm{G}$ be a diagonal matrix with distinct positive entries and $\bm{U}\in O(3)$. Then, $\Phi({\bm{Q}})=\text{tr}(\bm{G}(\bm{I}_3-\bm{U}^\top{\bm{Q}}\bm{U}))$ has four isolated critical points given by
\begin{equation*}{\bm{Q}}\in \{\bm{I}_3,\bm{U}\bm{D}_1\bm{U}^\top,\bm{U}\bm{D}_2\bm{U}^\top,\bm{U}\bm{D}_3\bm{U}^\top\},
\end{equation*}
where $\bm{D}_i=2[\bm{I}_3]_i[\bm{I}_3]_i^\top-\bm{I}_3$ and $[\bm{I}_3]_i$ is the $i$-th column vector of $\bm{I}_3$.
\end{Lemma}
\begin{Lemma}\cite{Quoc2020auto}\label{lm:characteristics_stationary_points}
Consider critical points of $\Phi({\bm{Q}})=\text{tr}(\bm{G}(\bm{I}_3-\bm{U}^\top{\bm{Q}}\bm{U}))$ in Lemma \ref{lm:stationary_points}. The desired critical point ${\bm{Q}}=\bm{I}_3$ is the global minimum of $\Phi$ while, depending on the values of the diagonal entries of $\bm{G}$, one of the three undesired points is the global maximum and the other two critical points are saddle points of $\Phi$. 
\end{Lemma}
Define the following error vectors:
\begin{align*}
\bm{e}_{21}&=\bm{b}_{12}^1\times \bm{b}_{21}^2,~
\bm{e}_{2n_l}=\bm{b}_{1n_l}^1\times\bm{b}_{2n_l}^2,~\forall x_l\in\mc{V}_a\\
\bm{e}_2&=k_{21}\bm{e}_{21}+\sum_{l\in\mc{V}_a}k_{2l}\bm{e}_{2n_l}. \numberthis \label{eq:error_vector}
\end{align*}
Then we have the following lemma whose proof is given in Appendix \ref{app:properties_error_functions12}.
\begin{Lemma}\label{lm:properties_error_functions12} The error function $\Phi_{2}$ and error vectors satisfy the following properties:
\begin{enumerate}[(i)]
\item $\dot{\Phi}_{2}=\bm\omega_2^2\cdot \bm{e}_2$
\item $||\dot{\bm{e}}_{2}||\leq \Big(k_{21}+\sum_{x_l\in\mc{V}_a}k_{2x_l}\Big)\abs{\bm\omega_2^2}^2$.
\item $\sigma_{2}||\bm{e}_2||^2\leq\Phi_{2}\le \gamma_{2}||\bm{e}_2||^2$ for some constants $\sigma_{2}, \gamma_{2}>0$, where the upper bound holds when $\Phi_{2}<2\min \{\lambda_j+\lambda_k\},~(\lambda_j,\lambda_k\in\lambda(\bm{K}_2),j\neq k)$.
\end{enumerate}
\end{Lemma}

\subsubsection{Orientation Alignment Laws} We now design the control law for agent $2$ as follows.
\begin{align}
\dot{\bm{\omega}}_2^2&=-k_\omega\bm{\omega}_2^2-\bm{e}_2,\label{eq:control_law_2}
\end{align}
where $k_\omega>0$ is a constant control gain. Then, the following result is obtained.
\begin{Theorem}\label{thm:stability_of_agent12}
Suppose that Assumptions \ref{ass:direction_measurements}-\ref{ass:non_colocated} hold. Then, under the orientation alignment law \eqref{eq:control_law_2}, for almost all positive scalars in \eqref{eq:error_function12}, we have:
\begin{enumerate}[(i)]\label{thm:orient_agent_3}
\item The invariant set of the system is given as $\{(\tilde{\bm{Q}}_{2},$ $\bm\omega_2^2):\tilde{\bm{Q}}_{2}\in \{\bm{I}_3,\bm{U}\bm{D}_1\bm{U}^\top,\bm{U}\bm{D}_2\bm{U}^\top,$ $\bm{U}\bm{D}_3\bm{U}^\top\},$ $\bm\omega_2^2=\bm{0}\}$, where $\bm{D}_i$ and $\bm{U}$ are defined in Lemma \ref{lm:stationary_points}.
\item The desired equilibrium $(\bm{I}_3,\bm{0})$ is almost globally asymptotically stable and locally exponentially stable and the three undesired equilibria are unstable.
\end{enumerate}
\end{Theorem}
\begin{proof} See Appendix \ref{app:stability_of_agent12}.
\end{proof}
\subsection{Agent $i\in \mc{V}\setminus \{1,2\}$}
Due to the cascade structure of the leader-follower system, we will use an induction argument to show the stability of the other agents. For induction, we assume that the corresponding results in Theorem \ref{thm:stability_of_agent12} hold for the agents $2,\ldots,i-1$, i.e., $\bm{R}_j\rightarrow\bm{R}_1$ as $t\rightarrow \infty,$ $\forall j\in\{2,\ldots,i-1\}$. 

Recall that agent $i$ measures the directions $(\bm{b}_{ij}^i,\bm{b}_{ik}^i)$ to its two neighbors, say $\{j,k\}$, and receives $(\bm{b}_{ji}^j,\bm{b}_{ki}^k)$ transmitted from the neighbors. We consider the normalized cross product, $\bm{b}_{in}^i:=(\bm{b}_{ij}^i\times\bm{b}_{ik}^i)/\abs{\bm{b}_{ij}^i\times\bm{b}_{ik}^i}$, as the third direction constraint for the proposed algorithm to work (i.e., for positive definite of $\bm{K}_i$ which will be defined after Eq. \eqref{eq:unforced-system_i} below). In other words, associated with each agent $i=3,\ldots,n$, a virtual neighbor, $n$, at the direction of $\bm{b}_{in}^i$ is considered, i.e., $\mc{N}_i=\{j,k,n\}$. To proceed, similar to the error vectors of agent $2$ in \eqref{eq:error_vector}, we define the error vector associated with each agent $i$ as follows.
\begin{equation}\label{eq:error_vector_i}
\bm{e}_i=\sum_{j\in\mc{N}_i}k_{ij}(\bm{b}_{ji}^j\times \bm{b}_{ij}^i).
\end{equation}
\subsubsection{Proposed alignment law} We propose the orientation alignment law for agent $i$ as
\begin{align*}
\dot{\bm{\omega}}_i^i&=-k_\omega\bm{\omega}_i^i-\bm{e}_i \numberthis\label{eq:control_law_i}\\
&=-k_\omega\bm{\omega}_i^i-\sum_{j\in\mc{N}_i}k_{ij}(\bm{b}_{ji}^1\times \bm{b}_{ij}^i)+\bm{h}_i(t), \numberthis \label{eq:cascade_system_i}
\end{align*}
where $\bm{b}_{ji}^1=\bm{R}_1^\top\bm{b}_{ji}$ denotes the directional vector $\bm{b}_{ji}$ expressed in the leader's coordinate system $^1\Sigma$, $\bm{R}_1$ is the orientation matrix of agent $1$, and 
\begin{equation*}
\bm{h}_i(t):=\sum_{j\in\mc{N}_i}k_{ij}\big((\bm{b}_{ji}^1-\bm{b}_{ji}^j)\times \bm{b}_{ij}^i\big).
\end{equation*} Since the right hand side of \eqref{eq:control_law_i} is linear in $\bm{\omega}_i^i$ and $\bm{e}_i$ is bounded, $\bm{\omega}_i^i$ is bounded and uniformly continuous in $t$.
Note that $\bm{h}_i(t)$ can be considered as an additive input to the following unforced system:
\begin{equation}\label{eq:unforced-system_i}
\dot{\bm{R}}_i=\bm{R}_i(\bm\omega_i^i)^\wedge,~\dot{\bm{\omega}}_i^i=-k_\omega\bm{\omega}_i^i-\bar{\bm{e}}_i,
\end{equation}
where $\bar{\bm{e}}_i:=\textstyle\sum_{j\in\mc{N}_i}k_{ij}(\bm{b}_{ji}^1\times \bm{b}_{ij}^i)$.
 
Let the symmetric matrix $\bm{K}_i=\sum_{j\in\mc{N}_i}k_{ij}\bm{b}_{ij}\bm{b}_{ij}^\top$, which has distinct eigenvalues for almost all $k_{ij}>0$ (Lemma \ref{lm:almost_surely_distinct}), and the error function
\begin{align*}
\Phi_i&=\sum_{j\in\mc{N}_i}k_{ij}(1+\bm{b}_{ji}^1\cdot \bm{b}_{ij}^i)\\
&=\text{tr}(\bm{G}_{i}(\bm{I}_3-\bm{U}^\top\tilde{\bm{Q}}_{i}\bm{U}),
\end{align*}
where $\tilde{\bm{Q}}_{i}:=\bm{R}_i\bm{R}_1^\top$, and $\bm{G}_i:=\text{diag}\{\lambda(\bm{K}_{i})\}$ is the diagonal matrix associated with the decomposition $\bm{K}_i=\bm{U}\bm{G}_i\bm{U}^\top$. Thus, $\Phi_i$ has four isolated critical points (Lemma \ref{lm:stationary_points}) which satisfy the properties in Lemma \ref{lm:characteristics_stationary_points}. Furthermore,
we have the following lemma whose proof is similar to the proof of Lemma \ref{lm:properties_error_functions12}. 
\begin{Lemma}
The error function $\Phi_i$ and error vector $\bar{\bm{e}}_i$ satisfy the following properties:
\begin{enumerate}[(i)]
\item $\dot{\Phi}_i=\bar{\bm{e}}_i\cdot\bm{\omega}_i^i$.
\item $||\dot{\bar{\bm{e}}}_i||\leq \sum_{j\in\mc{N}_i}k_{ij}||\bm{\omega}_i^i||$.
\item $\sigma_{i}||\bar{\bm{e}}_i||^2\leq\Phi_{i}\le \gamma_{i}||\bar{\bm{e}}_i||^2$ for some constants $\sigma_{i}, \gamma_{i}>0$, where the upper bound holds when ${\Phi}_{i}<2\min \{\lambda_j+\lambda_k\},~(\lambda_j,\lambda_k\in\lambda(\bm{K}_{i}),j\neq k)$.
\end{enumerate}
\end{Lemma}
\subsubsection{Stability Analysis} We have the following theorem whose proof is similar to Proof of Theorem \ref{thm:stability_of_agent12}.
\begin{Theorem}\label{thm:unforced_system_i}
Suppose that Assumptions \ref{ass:direction_measurements}-\ref{ass:non_colocated} hold. Then, under the orientation alignment law \eqref{eq:control_law_i}, for almost all positive scalars $k_{ij}>0$, there holds:
\begin{enumerate}[(i)]
\item The invariant set of the unforced system \eqref{eq:unforced-system_i} is given as $\{(\tilde{\bm{Q}}_{i},\bm\omega_i^i):\tilde{\bm{Q}}_{2}\in \{\bm{I}_3,\bm{U}\bm{D}_1\bm{U}^\top,\bm{U}\bm{D}_2\bm{U}^\top,$ $\bm{U}\bm{D}_3\bm{U}^\top\},$ $\bm\omega_i^i=\bm{0}\}$, where $\bm{D}_i$ are defined in Lemma \ref{lm:stationary_points}.
\item The desired equilibrium $(\bm{I}_3,\bm{0})$ is almost globally asymptotically stable and locally exponentially stable and the three undesired equilibria are unstable.
\end{enumerate}
\end{Theorem}
\begin{Lemma}\label{lm:vanishing_input_i}
The input $\bm{h}_i(t)$ in \eqref{eq:cascade_system_i} is bounded and vanishes asymptotically as $t\rightarrow\infty$.
\end{Lemma}
\begin{proof}
The input $\bm{h}_i(t)$ is clearly bounded. Furthermore, because $\bm{R}_j,j=1,\ldots,i-1$ tend to $\bm{R}_1$ as $t\rightarrow\infty$ we have that $(\bm{b}_{ji}^1-\bm{b}_{ji}^j)\rightarrow\bm{0}$, for all $j\in \mc{N}_i$. This completes the proof.
\end{proof}
\begin{Lemma}\label{lm:aISS_i}
The cascade system \eqref{eq:cascade_system_i} is almost globally input-to-state stable with respect to the input $\bm{h}_i(t)$ when $\lambda_{min}(\bm{K}_i)\approx \lambda_{max}(\bm{K}_i)$ and $k_{ij}$ are sufficiently large, with respect to $\bm{h}_i(t)$.
\end{Lemma}
\begin{proof}
Consider the Lyapunov function:
\begin{equation}
V_i=\Phi_i+\frac{1}{2}||\bm{\omega}_i^i||^2+k_V\bm{\omega}_i^i\cdot\bar{\bm{e}}_i.
\end{equation}
Let $\bm{z}_i=[\bar{\bm{e}}_i,||\bm{\omega}_i^i||]^\top$; we then have
\begin{equation}\label{eq:psd_V_i}
\frac{1}{2}\bm{z}_i^\top\left[\begin{matrix}
2\sigma_{i} &-k_V\\
-k_V &1
\end{matrix}\right]\bm{z}_i\leq V_i \leq\frac{1}{2}\bm{z}_i^\top\left[\begin{matrix}
2\gamma_{i} &k_V\\
k_V &1
\end{matrix}\right]\bm{z}_i.
\end{equation}
It is noted that $V_i$ is positive definite for a sufficient small $k_V$. The derivative of $V_i$ along the trajectory of \eqref{eq:cascade_system_i} is given as
\begin{align*}
\dot{\bm{V}}_i&=-k_\omega||\bm{\omega}_i^i||^2+\bm{\omega}_i^i\cdot\bm{h}_i+k_V\bar{\bm{e}}_i^\top(-k_\omega\bm{\omega}_i^i-\bar{\bm{e}}_i+\bm{h}_i)\\
&\quad+k_V\bm{\omega}_i^i\cdot\dot{\bar{\bm{e}}}_i\\
&\leq-(k_\omega-k_V\sum_{j\in\mc{N}_i}k_{ij})||\bm{\omega}_i^i||^2-k_V||\bar{\bm{e}}_i||^2\\
&\quad-k_Vk_\omega\bar{\bm{e}}_i^\top\bm{\omega}_i^i+||\bm{\omega}_i^i+k_V\bar{\bm{e}}_i||||\bm{h}_i||\\
&\leq -\frac{1}{2}\bm{z}_i^\top\bm{P}_k\bm{z}_i+\delta_k||\bm{h}_i||
\end{align*}
where $\delta_k:=\sup_{t}||\bm{\omega}_i^i+k_V\bar{\bm{e}}_i||$, which is finite due to boundedness of $\bm{\omega}_i^i$ and $\bar{\bm{e}}_i$, and
$$
\bm{P}_k=
\begin{bmatrix}
2k_V &k_Vk_\omega\\
k_Vk_\omega &2(k_\omega-k_V\sum_{j\in\mc{N}_i}k_{ij})
\end{bmatrix},
$$
which is positive definite when $k_V<\frac{4k_\omega}{4\sum_{j\in\mc{N}_i}k_{ij}+k_{\omega}^2}$. Let $\bm{M}_k$ and $\bm{N}_k$ respectively be the matrices in the left hand side and right hand side of \eqref{eq:psd_V_i}. Then, we obtain
\begin{equation}\label{eq:local_ISS_i}
\dot{V}_i\leq-\frac{\lambda_{min}(\bm{P}_k)}{\lambda_{max}(\bm{N}_k)} V_i+\delta_k||\bm{h}_i||
\end{equation}
which shows boundedness property of \eqref{eq:cascade_system_i} \cite[Prop. 3]{Angeli2011tac}. Thus, the cascade system is (locally) input-to-state stable wrt. $\bm{h}_i(t)$. 

When $\Phi_i<\phi_i:=2\min \{\lambda_1+\lambda_2,\lambda_1+\lambda_3,\lambda_2+\lambda_3\}$, following \eqref{eq:psd_V_i} we have that
\begin{align*}
\Phi_i&\leq \gamma_{i}||\bm{e}_i||^2\leq \gamma_i||\bm{z}_i||^2\\
&\leq \frac{2\gamma_i}{\lambda_{min}(\bm{M}_i)}V_i\\
&\leq\frac{2\gamma_i\lambda_{max}(\bm{N}_k)}{\lambda_{min}(\bm{M}_i)}||\bm{z}_i||^2.
\end{align*}
Let $\kappa_i:=\frac{2\gamma_i\lambda_{max}(\bm{N}_k)}{\lambda_{min}(\bm{M}_i)}$. It follows that if we initially have $||\bm{z}_i(0)||^2<\phi_i/\kappa_i$ then the condition $\Phi_i<\phi_i$ is satisfied. In addition, if $k_{ij}$ are selected such that $\lambda_{min}(\bm{K}_i)\approx \lambda_{max}(\bm{K}_i)$, then $\tilde{\bm{Q}}_i(0)$ that satisfies $\Phi_i(\tilde{\bm{Q}}_i(0))<\phi_i$ covers almost all $SO(3)$, and $\bm{\omega}_i^i(0)$ satisfies $||\bm{\omega}_i^i(0)||^2<\phi_i/\kappa_i-||\bar{\bm{e}}_i||^2<\phi_i/\kappa_i-\Phi_i(0)/\gamma_i$ cover $\bm{R}^3$ when $k_{ij}\rightarrow \infty,\forall j\in \mc{N}_i$. This completes the proof.
\end{proof}
\begin{Remark}
Agent $i$ can locally design $\lambda(\bm{K}_i)$ because $\bm{K}_i$ is similar to the matrix $$\bm{K}_i^i:=\sum_{j\in\mc{N}_i}k_{ij}\bm{b}_{ij}^i(\bm{b}_{ij}^i)^\top.$$
Note that the third direction in the set $\{\bm{b}_{ij}\},j\in\mc{N}_i:=\{k,l,m\}$ is orthogonal to the first two directional vectors (due to its construction). It follows that $k_{im}$ is an eigenvalue of $\bm{K}_i^i$ corresponding to the eigenvector $\bm{b}_{im}^i$ because
\begin{equation*}
\bm{K}_i^i\bm{b}_{im}^i=\sum_{j\in\mc{N}_i}k_{ij}\bm{b}_{ij}^i(\bm{b}_{ij}^i)^\top\bm{b}_{im}^i=k_{im}\bm{b}_{im}^i.
\end{equation*}
The other two eigenvalues can be obtained by inspecting the roots of the characteristic equation $\text{det}(\lambda\bm{I}_3-\bm{K}_i^i)=(\lambda-k_{im})f(\lambda)$, where $f(\lambda)$ is a quadratic function. Thus, agent $i$ can easily choose $k_{ik}$ and $k_{il}$ such that two roots of $f(\lambda)=0$ arbitrary close to $\lambda=k_{im}$.
\end{Remark}
\begin{Theorem}
Suppose that Assumptions \ref{ass:direction_measurements} and \ref{ass:non_colocated} hold. Then, under the orientation alignment law \eqref{eq:control_law_i}, $\bm{R}_i\rightarrow\bm{R}_1$ almost globally and asymptotically as time goes to infinity.
\end{Theorem}
\begin{proof}
The desired equilibrium $(\tilde{\bm{Q}}_{i}=\bm{I}_3,\bm\omega_i^i=\bm{0})$ of the unforced system \eqref{eq:unforced-system_i} is almost globally asymptotically stable (Theorem \ref{thm:unforced_system_i}). The input $\bm{h}_i(t)$ is bounded and converges to zero asymptotically by Lemma \ref{lm:vanishing_input_i} and the cascade system \eqref{eq:cascade_system_i} is input-to-state stable with respect to $\bm{h}_i(t)$ (Lemma \ref{lm:aISS_i}). Thus, the desired equilibrium $(\tilde{\bm{Q}}_{i}=\bm{I}_3,\bm\omega_i^i=\bm{0})$ of the cascade system \eqref{eq:cascade_system_i} is almost globally asymptotically stable \cite{Angeli2011tac}.
\end{proof}

Finally, by invoking induction, we conclude that $\bm{R}_i$ converges to $\bm{R}_1$, or equivalently, the $i$-th local coordinate system aligns to that of the first agent, almost globally and asymptotically for all $i\in \mc{V}$.
\section{Simulation}\label{sec:simulation}

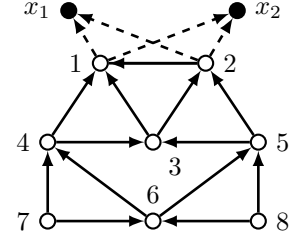
\begin{figure}[t]
\centering
\begin{tikzpicture}[scale=1.4]
\node[place] (p1) at (0,0) [label=left:$1$]{};
\node[place] (p2) at (1,0) [label=right:$2$]{};
\node[place] (p3) at (0.5,-0.75) [label=below right:$3$]{};
\node[place] (p4) at (-0.5,-0.75) [label=left:$4$]{};
\node[place] (p5) at (1.5,-0.75) [label=right:$5$]{};
\node[place] (p6) at (.5,-1.5) [label=above:$6$]{};
\node[place] (p7) at (-.5,-1.5) [label=left:$7$]{};
\node[place] (p8) at (1.5,-1.5) [label=right:$8$]{};
\node[place, black] (x1) at (-0.3,0.5) [label=left:$x_1$]{};
\node[place, black] (x2) at (1.3,0.5) [label=right:$x_2$]{};

\draw[line width=1pt,->] (p3)[frame]  -- (p1);
\draw[line width=1pt,->] (p3)[frame]  -- (p2);
\draw[line width=1pt,->] (p4)[frame]  -- (p3);
\draw[line width=1pt,->] (p4)[frame]  -- (p1);
\draw[line width=1pt] (p2)  -- (p1);
\draw[line width=1pt,->] (p5)[frame]  -- (p3);
\draw[line width=1pt,->] (p5)[frame]  -- (p2);
\draw[line width=1pt,->] (p6)[frame]  -- (p4);
\draw[line width=1pt,->] (p6)[frame]  -- (p5);
\draw[line width=1pt,->, dashed] (p1)[frame]  -- (x1);
\draw[line width=1pt,->, dashed] (p2)[frame]  -- (x1);
\draw[line width=1pt,->, dashed] (p1)[frame]  -- (x2);
\draw[line width=1pt,->, dashed] (p2)[frame]  -- (x2);
\draw[line width=1pt,->] (p2)[frame]  -- (p1);
\draw[line width=1pt,->] (p7)[frame]  -- (p4);
\draw[line width=1pt,->] (p7)[frame]  -- (p6);
\draw[line width=1pt,->] (p8)[frame]  -- (p6);
\draw[line width=1pt,->] (p8)[frame]  -- (p5);
\end{tikzpicture}
\caption{The graph topology $\mc{G}(\mc{V},\mc{E})$ of eight agents and two landmarks $x_1$ and $x_2$.}
\label{fig:sim_graph}
\end{figure}

\begin{figure*}[t]
\begin{subfigure}[t]{0.3\textwidth}
\centering
\includegraphics[height=4.2cm]{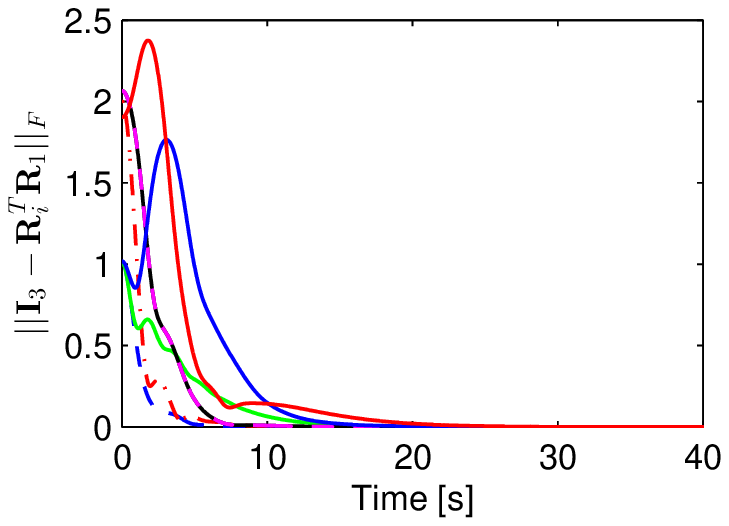}
\caption{}
\label{fig:orient_err}
\end{subfigure}
\centering
\begin{subfigure}[t]{0.34\textwidth}
\centering
\includegraphics[height=5.5cm]{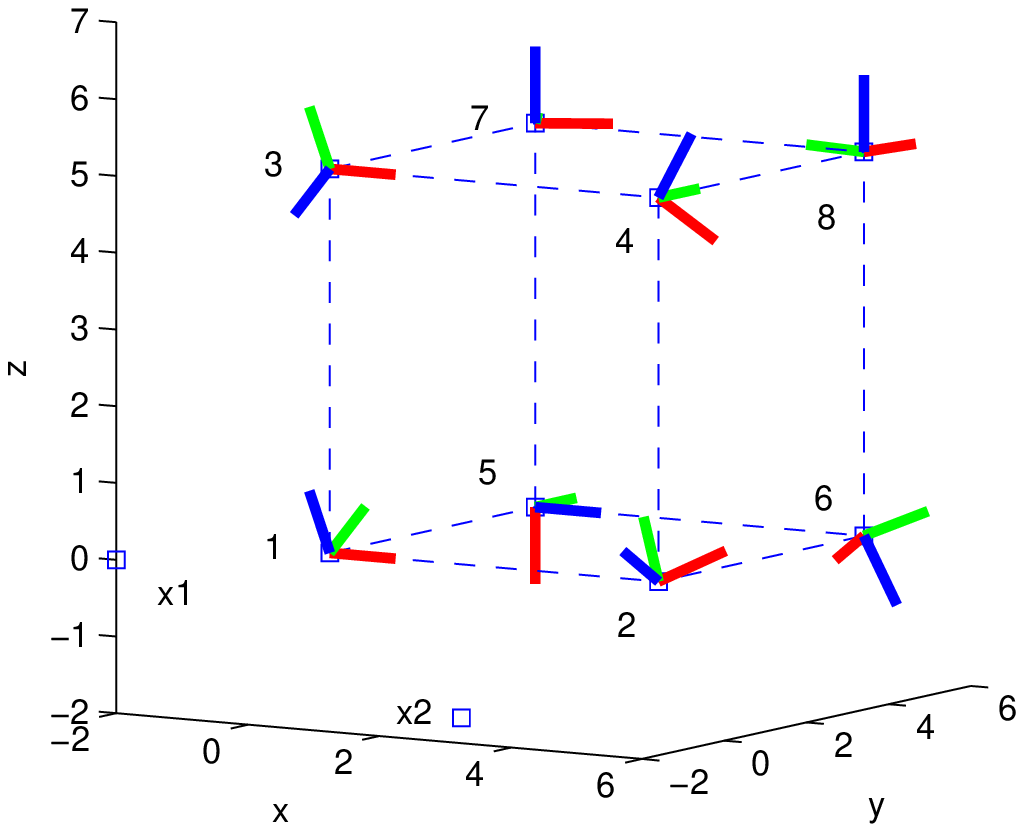}
\caption{}
\label{fig:initial_orient}
\end{subfigure}
\begin{subfigure}[t]{0.34\textwidth}
\centering
\includegraphics[height=5.5cm]{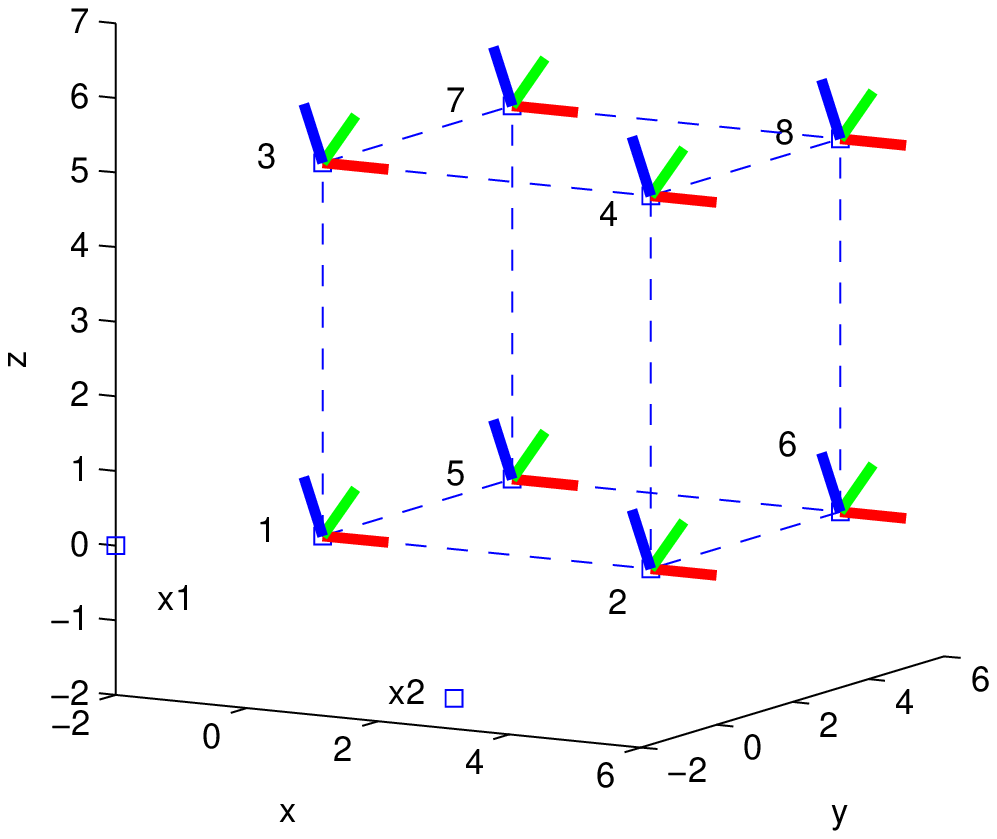}
\caption{}
\label{fig:final_orient}
\end{subfigure}
\caption{Orientation alignment of eight agents in $\mb{R}^3$. (a) Orientation alignment errors $||\bm{I}_3-{\bm{R}}_{i}^\top\bm{R}_1||_\text{F}$. (b) Initial orientations of the agents (local coordinate axes are shown in red-green-blue). (c) Final orientations.}
\label{fig:sim}
\end{figure*}
Consider a leader-follower network of eight agents ($n=8$) and two landmarks in $\mb{R}^3$ whose interaction graph is depicted in Fig. \ref{fig:sim_graph}. The positions and initial orientations, whose coordinate axes are shown in red-green-blue, of the agents are depicted in Fig. \ref{fig:initial_orient}, respectively. In particular, the initial orientations of the agents are chosen as: $\bm{R}_1=\mathrm{rotX}(\pi/6),\bm{R}_2=\mathrm{rotX}(\pi/3)\mathrm{rotZ}(\pi/6),\bm{R}_3=\mathrm{rotX}(2\pi/3),\bm{R}_4=\mathrm{rotY}(\pi/6),\bm{R}_5=\mathrm{rotY}(\pi/2),\bm{R}_6=\mathrm{rotY}(5\pi/6)\mathrm{rotZ}(\pi/6), \bm{R}_7=\mathrm{rotZ}(\pi/6)$, and $\bm{R}_8=\mathrm{rotZ}(8\pi/9)$, where $\mathrm{rotX}(\cdot),\mathrm{rotY}(\cdot)$ and $\mathrm{rotZ}(\cdot)$ denote the rotations about $x$, $y$ and $z$ axes, respectively. It is observed that the orientations of the agents converge to the orientation of the agent $1$ asymptotically as the orientation alignment errors $||\bm{I}_3-\bm{R}_{i}^\top\bm{R}_1||_\text{F}$, where $||\cdot||_\text{F}$ denotes the Frobenius norm, tend to zeros asymptotically as $t\rightarrow\infty$, as shown in Fig. \ref{fig:orient_err}.

The final orientations of the agents are shown in Fig. \ref{fig:final_orient}. A video of the simulation is available at \url{https://youtu.be/amP5svWhfrU}.
\section{Conclusion}\label{sec:Conclusion}
In this work, we have studied the orientation alignment of a type of leader-follower networks of multiple stationary agents with the first agent being the leader in the three-dimensional space. The agents in the network have only direction measurements to two neighbors and the first two agents can additionally measure directional vectors to some common landmarks. An orientation alignment scheme was proposed for the agents based exclusively in the direction constraints with no need for a global coordinate frame common to all agents or the construction of the agents' orientation matrices. We showed that the agents can achieve an orientation consensus almost globally and asymptotically. Simulation results were provided to illustrate the theoretical analysis.

A potential research direction is to investigate the formation control and formation maneuvering of multi-agent systems based on the direction-only measurements and possibly in the presence of some leader agents with global information. 
Direction-only orientation alignment for more practical agent's dynamics with arbitrary pre-specified time stability will be also addressed in a future work.
 
\appendix
\renewcommand{\thesection}{A\arabic{section}}
%
%
\subsection{Proof of Lemma \ref{lm:almost_surely_distinct}}\label{app:almost_surely_distinct}
It is noted that each term $\bm{b}_{2j}\bm{b}_{2j}^\top,~j\in \mc{N}_2\cup\mc{V}_a$ in $\bm{K}_{2}$ is positive semi-definite whose range is $\text{range}(\bm{b}_{2j}\bm{b}_{2j}^\top)=\text{span}(\bm{b}_{2j})$. In addition, because agents $1$, $2$ and the landmarks are non-coplanar (Assumption \ref{ass:non_colocated}), $\{\bm{b}_{21},\{\bm{b}_{2n_l}\}_{x_l\in\mc{V}_a}\}$ spans the $\mb{R}^3$. It follows that $\bm{K}_{2}$ is positive definite. 

Since $\bm{K}_{2}$ is symmetric and positive definite, it has positive real eigenvalues.
Consider the discriminant of the cubic polynomial $\text{det}(\lambda\bm{I}_3-\bm{K}_{2})$. It is noted that the discriminant has real coefficients, and is a polynomial of the scalars $k_{21}$ and $k_{2x_l}$. Futhermore, since all eigenvalues of $\bm{K}_{2}$ are real the discriminant is nonnegative \cite[Chap. 10.3]{Irving2004}, and a fundamental result of algebraic geometry indicates that the set of the scalars such that the discriminant is zero is of measure zero \cite{Caron2005}. Consequently, the discriminant of $\text{det}(\lambda\bm{I}_3-\bm{K}_{2})$ is positive for almost all scalars $k_{21}$ and $k_{2x_l}$, and hence $\bm{K}_{2}$ has distinct eigenvalues for almost all scalars $k_{21}$ and $k_{2x_l}$ \cite[Chap. 10.3]{Irving2004}.

\subsection{Proof of Lemma \ref{lm:properties_error_functions12}}\label{app:properties_error_functions12}
Consider the derivative of $\Psi_{12}:=(1+\bm{b}_{12}^1\cdot\bm{b}_{21}^2)$ as follows:
\begin{align*}
\dot{\Psi}_{12}&=\dot{\bm{b}}_{12}^1\cdot\bm{b}_{21}^2+\bm{b}_{12}^1\cdot\dot{\bm{b}}_{21}^2\stackrel{\eqref{eq:derivative_of_direction}}{=}\bm{b}_{12}^1\cdot(\bm{b}_{21}^2\times \bm{\omega}_2^2)\\
&\stackrel{\eqref{eq:cross_prod_4}}{=}\bm{\omega}_2^2\cdot (\bm{b}_{12}^1\times \bm{b}_{21}^2)=\bm\omega_2^2\cdot\bm{e}_{21},
\end{align*}
By using similar calculations for the other terms and substituting into the time derivative of $\Phi_2$ in \eqref{eq:error_function12} we obtain (i).

We show (ii) as follows. Consider the time derivative of $\bm{e}_{21}$ as follows.
\begin{align*}
\dot{\bm{e}}_{21}&=\bm{b}_{12}^1\times\bm{b}_{21}^2\times\bm{\omega_2^2}\\
\Longleftrightarrow ||\dot{\bm{e}}_{21}||&\leq \abs{\bm\omega_2^2}^2.
\end{align*}
It can be shown similarly that $||\dot{\bm{e}}_{in_k}||\leq||\bm\omega_2^2||^2,\forall x_k\in \mc{V}_a$. Consequently,
$$||\dot{\bm{e}}_2||\leq \Big(k_{21}+\sum_{l\in\mc{V}_a}k_{2x_l}\Big)\abs{\bm\omega_2^2}^2.$$

The proof of (iii) follows from similar arguments in \cite[Lem. 3]{Quoc2019autoExtend} and \cite[Prop. 1]{TLee2015tac} and is omitted.

\subsection{Proof of Theorem \ref{thm:stability_of_agent12}}\label{app:stability_of_agent12}
Consider the Lyapunov candidate function
\begin{align*}\label{eq:Lyapunov_function12}
V_{2}&=\Phi_{2}+\frac{1}{2}||\bm\omega_2^2||^2+k_V(\bm{e}_2^\top\bm\omega_2^2) \numberthis
\end{align*}
Following Lemma \ref{lm:properties_error_functions12} (iii) we have
\begin{align} \label{eq:positive_def_V12}
\frac{1}{2}\bm{z}_2^\top\left[\begin{matrix}
\sigma_{2} &-k_V\\
-k_V &1
\end{matrix}\right]\bm{z}_2\leq V_{2}\leq
\frac{1}{2}\bm{z}_2^\top\left[\begin{matrix}
\gamma_{2} &k_V\\
k_V &1
\end{matrix}\right]\bm{z}_2,
\end{align}
where $\bm{z}_2:=[||\bm{e}_2||,||\bm\omega_2^2||]^\top$. We note that if $k_V<\sqrt{\sigma_{2}}$ both matrices on the left and right hand sides of \eqref{eq:positive_def_V12} are positive definite and it is also true for $V_{2}$. The derivative of $V_{2}$ along the trajectory of \eqref{eq:control_law_2} is given as
\begin{align*}
\dot{V}_{2}&=(\bm\omega_2^2)^\top \bm{e}_2-k_\omega||\bm\omega_2^2||^2-\bm{e}_2^\top\bm\omega_2^2
+k_V(\dot{\bm{e}}_2^\top\bm\omega_2^2+\bm{e}_2^\top\dot{\bm\omega}_2^2)
\\
&\leq-k_\omega||\bm\omega_2^2||^2+k_V(k_{2}^{tot}||\bm\omega_2^2||^2-k_\omega\bm{e}_2^\top{\bm\omega}_2^2-\bm{e}_2^\top\bm{e}_2)\\
&=-(k_\omega-k_Vk_{2}^{tot})||\bm\omega_2^2||^2-k_V||\bm{e}_2||^2-k_Vk_\omega\bm{e}_2^\top\bm\omega_2^2\\
&\leq -\frac{1}{2}\bm{\xi}_{2}^\top\left[\begin{matrix}
2k_V &k_Vk_\omega \\
k_Vk_\omega &2(k_\omega-k_Vk_{2}^{tot})
\end{matrix}\right]\bm{\xi}_{2},
\end{align*}
where $\bm{\xi}_{2}=[||\bm{e}_2||,||\bm\omega_2^2||]^\top\in \mb{R}^2$, and $k_{2}^{tot}=k_{21}+\sum_{l\in\mc{V}_a}k_{2x_l}$. It can be verified that if $k_V<\frac{4k_\omega}{4k_{2}^{tot}+k_{\omega}^2}$ the matrix in the above inequality is positive definite. It follows that $\dot{V}_{2}$ is negative definite. As a result, $\bm{\xi}_{2}\rightarrow \bm{0}$ as $t\rightarrow \infty$ according to LaSalle's invariance principle. As a result, as $\bm{e}_2\rightarrow \bm{0}$ and $\bm\omega_2^2\rightarrow \bm{0}$ as $t\rightarrow \infty$, we have $\dot{\Phi}_{2}\rightarrow 0$ by Lemma \ref{lm:properties_error_functions12}(i). Consequently, the equilibrium points of the system satisfy $\bm\omega_2^2=\bm{0}$ and $\tilde{\bm{Q}}_{2}$ are critical points of $\Phi_{2}$, which shows (i).

Using Lemma \ref{lm:characteristics_stationary_points} we show that the three undesired equilibria are unstable as follows. Without loss of generality, consider the first undesired equilibrium point at which we have
\begin{equation*}
\Phi_{2}(\bm{U}\bm{D}_1\bm{U}^\top)=\text{tr}(\bm{G}_{2}(\bm{I}_3-\bm{D}_1))=2(\lambda_2+\lambda_3).
\end{equation*}
Consider the Lyapunov function $U_{2}=2(\lambda_2+\lambda_3)-V_{2}$,
which satisfies $U_2(\bm{U}\bm{D}_1\bm{U}^\top,~\bm\omega_2^2=\bm{0})=0.$ Due to the continuity of $U_2$, when $\bm\omega_2^2$ is sufficiently small we can select $\tilde{\bm{Q}}_{2}$ arbitrary close to the undesired point $\bm{U}\bm{D}_1\bm{U}^\top$ (which is either a global maximum or saddle point of $\Phi_{2}$ by Lemma \ref{lm:characteristics_stationary_points}) such that $U_{2}>0$. However, $=\dot{U}_2=-\dot{V}_{2}>0$. It follows that the undesired equilibrium $(\tilde{\bm{Q}}_{2}=\bm{U}\bm{D}_1\bm{U}^\top,~\bm\omega_2^2=\bm{0})$ is unstable due to the Chetaev's theorem \cite[Thm. 4.3]{Khalil2002}.
Consequently, the desired equilibrium of the system is globally asymptotically stable except on a set of measure zero which contains stable manifolds of the undesired equilibria.

The desired equilibrium point is locally exponentially stable in the region satisfying $\Phi_{2}<\phi:=2\min \{\lambda_1+\lambda_2,\lambda_1+\lambda_3,\lambda_2+\lambda_3\}$.

\section*{Acknowledgments}

This work is supported by the BK21 FOUR Program of the National Research Foundation Korea (NRF) grant funded by the Ministry of Education(MOE), in part by the Future Mobility Testbed Development through IT, AI, and Robotics, and in part by the National Research Foundation (NRF) of Korea under the grant NRF2017R1A2B3007034.


\nocite{}
\bibliographystyle{IEEEtran}
\bibliography{quoc2018,quoc2019}

\end{document}